\documentclass[11pt,a4paper,oneside,onecolumn,fleqn]{article}
\usepackage{mathrsfs}
\usepackage{amsfonts}
\usepackage{longtable}
\usepackage{mathtools}
\usepackage{colortbl}
\usepackage{amssymb}

 \usepackage{rotating}
 \usepackage{multirow}

\usepackage{latexsym}
\usepackage{amsmath,amsthm}
\usepackage{epsf}
\usepackage{graphicx}
\usepackage{indentfirst}
\usepackage{cite}
\usepackage[numbers,sort&compress]{natbib}
\usepackage{extarrows}
\usepackage{caption2}

\renewcommand{\arraystretch}{1}

\theoremstyle{plain}
\newtheorem{thm}{\bf Theorem}[section]

\newtheorem{lem}[thm]{\bf Lemma}

\theoremstyle{definition}

\theoremstyle{remark}

\begin{document}
\baselineskip 18pt

\title{\bf Magic squares with all subsquares of possible orders based on extended Langford sequences}
\date{}
    \author{    \small     Wen Li$^{a}$, Ming Zhong$^{a}$, Yong Zhang$^{b,*}$   \\
                     \footnotesize { \it $^a$School of  Science, Xichang University, Sichuan 615013,  China} \ \ \ \    \ \ \  \ \ \ \ \ \ \ \  \ \ \ \ \  \ \ \ \ \ \ \ \ \ \ \ \ \ \ \ \ \ \ \ \  \\
                     \footnotesize {\it $^b$School of Mathematics and Statistics, Yancheng Teachers University, Jiangsu 224002,  China}\\
  }

 \maketitle

 \begin{abstract}
 \noindent
A magic square  of order $n$  with all subsquares of possible orders (ASMS$(n)$) is a magic square which  contains a general magic square of each order  $k\in\{3, 4, \cdots, n-2\}$. Since the conjecture on the existence of an ASMS was proposed in 1994, much attention has been paid but very little is known except for few sporadic examples. A $k$-extended Langford sequence of defect $d$ and length $m$ is equivalent to a partition of $\{1,2,\cdots,2m+1\}\backslash\{k\}$ into differences $\{d,\cdots,d+m-1\}$.  In this paper, a construction of ASMS based on  extended Langford sequence is established. As a result, it is shown that there exists an ASMS$(n)$ for $n\equiv\pm3\pmod{18}$, which gives a  partial answer to Abe's  conjecture on ASMS.  

 \end{abstract}

  {\begingroup\makeatletter\let\@makefnmark\relax\footnotetext{*Corresponding author: Y. Zhang(zyyctu@gmail.com, zyyctc@126.com). }
  {\begingroup\makeatletter\let\@makefnmark\relax\footnotetext{Research supported by the Natural Science Foundations of China (Nos.11301457,11371308).}

\section{Introduction}

    An $n\times n$ matrix $A$ consisting of $n^2$   integers is a {\it general magic square of order $n$},   denoted by  GMS$(n)$, if the sum of $n$ elements in each row, each column, main diagonal and back  diagonal is the same.  The sum is the \emph{magic number}.   A GMS$(n)$ is  a \emph{magic square}, denoted by  MS$(n)$,  if it consists of $n^2$ consecutive integers.  A lot of work has been done on   magic squares (\cite{Abe,Ahmed,Andrews,Handbook}).

 An  MS$(n)$  \emph{with all subsquares of possible orders}, denoted by  ASMS$(n)$, is an MS($n$) which contains a GMS$(k)$ for each integer $k$ such that $3\leq k\leq n-2$.

 The following is an ASMS$(6)$ (see \cite{Abe}) in which there is  a GMS$(3)$ in the lower right corner  and a GMS$(4)$ in the upper left corner.
    \begin{center}
  {\renewcommand\arraystretch{0.8}
 \setlength{\arraycolsep}{2.5pt}
 \small
$\left(\
      \begin{array}{ccccccccc     }
      \cline{1-4}
\multicolumn{1}{|c}{3}&6&34&\multicolumn{1}{c|}{28}&35&5\\
\multicolumn{1}{|c}{20}&31&2&\multicolumn{1}{c|}{18}&15&25\\
\multicolumn{1}{|c}{19}&12&26&\multicolumn{1}{c|}{14}&10&30\\\cline{4-6}
\multicolumn{1}{|c}{29}&22&9&\multicolumn{1}{|c|}{11}&13&\multicolumn{1}{c|}{27}\\   \cline{1-4}
8&36&16&\multicolumn{1}{|c}{33}&17&\multicolumn{1}{c|}{1}\\
32&4&24&\multicolumn{1}{|c}{7}&21&\multicolumn{1}{c|}{23}\\ \cline{4-6}
   \end{array}\   \right)
 $}
\end{center}

Abe gave the ASMS$(n)$ for $n=6,7,13,15,17$ and Kobayashi gave the ASMS$(n)$ for $n=8,9,10$ (\cite{Abe}).
Abe  \cite{Abe} also  proposed a conjecture on the existence of ASMS$(n)$ of odd $n$ as follows.

\vskip5pt
  {\bf Conjecture 2.12} (\cite{Abe}) There exists an  ASMS$(n)$ for every odd integer $n\geq5$.
\vskip5pt

In this paper,   Conjecture 2.12 is investigated and the extended Langford sequences are used to construct a quantity of submatrices (quasi-magic $2$-rectangle)  of an ASMS.

A \emph{$k$-extended Langford sequence of defect $d$ and length $m$}, denoted by $\mathcal{L}^k_{d,m}$, is a sequence $L=(l_1, l_2, \cdots, l_{2m+1})$ in which $l_k$ is empty, and each other member of the sequence comes from the set $\{d, d+1, \cdots, d+m-1\}$. Each $j\in \{d, d+1, \cdots, d+m-1\}$  occurs exactly twice in the sequence, and two occurrences are separated by exactly $j-1$ symbols.

An $\mathcal{L}^k_{d,m}$ is equivalent to a partition of $\{1,2,\cdots,2m+1\}\backslash\{k\}$ into differences $\{d,\cdots,d+m-1\}$. We refer the readers to the references \cite{Linek1998,Linek2004,Handbook,Francetic2009} for  (hooked, extended) Skolem squences and Langford squences and their applications such as cyclic Steiner systems \cite{Shalaby}, cyclic $m$-cycle systems \cite{Bryant},  and cyclically decomposing the complete graph into cycles \cite{Fua}.

 A  \emph{quasi-magic rectangle of size $(m,n)$}, denoted by QMR$(m,n)$, is an $m\times n$ array consisting of distinct
integers arranged so that the sum of the entries
of each row is a constant and each column sum is another  constant.
 A  \emph{ quasi-magic $3$-rectangle of size $(m,n,r)$}, denoted by QMR$(m,n,r)$,  is a set of $r$ QMR$(m,n)$s
 consisting of $mnr$ distinct integers such that if supposing the $r$ QMR$(m,n)$s then the sum of $r$ elements in each entry is the same.
 A  QMR$(m,n)$ consisting of consecutive integers is a \emph{magic rectangle}.  A   QMR$(m,n,r)$ consisting of consecutive integers is known as a   $3$-dimensional magic rectangle. We refer the readers to the references  \cite{Sun,ZLZS,Hagedorn2} for  details.

In this paper, we shall use extended Langford sequences and quasi-magic rectangles to prove the following.

  \begin{thm}\label{1.1}
 There exists an  ASMS$(n)$ for $n\equiv\pm3\pmod{18}$.
 \end{thm}

The rest of this paper is arranged as follows. Quasi-magic rectangles based on extended Langford sequences and  investigated in Section 2. Constructions and existence of ASMS  are provided in Section 3.

 \section{Preliminaries}

An integer set $\{a,a+1,\cdots,b\}$ is denoted as $[a,b]$. In this section we will introduce that an  $\mathcal{L}^k_{d,m}$   gives some special quasi-magic rectangles which are  the main blocks of an ASMS.

Let $L=(l_1, l_2, \cdots, l_{2m+1})$  be an   $\mathcal{L}^k_{d,m}$.
$L$ is  also written as a collection of ordered
pairs $\{(a_i, b_i)| i\in[d, d+m-1], b_i-a_i=i\}$ with $\bigcup_{i=d}^{d+m-1}\{a_i, b_i\} = [1, 2m+1]\setminus\{k\}$.
Adding $d+m-1$ to each number in the equality gives $\bigcup_{i=d}^{d+m-1}\{a_i+d+m-1, b_i+d+m-1\} = [d+m, d+3m]\setminus\{k+d+m-1\}$.
For each $i\in[d, d+m-1]$, adding $i$ to the pair $(a_i+d+m-1, b_i+d+m-1)$ as another coordinate gives
$\bigcup_{i=d}^{d+m-1}\{i, a_i+d+m-1, b_i+d+m-1\} = [d, d+3m]\setminus\{k+d+m-1\}$.
Let
\begin{center}
  $u_i=i, v_i=a_i+d+m-1, w_i=b_i+d+m-1$, $ i\in[d, d+m-1].$
\end{center}
It  gives a collection of ordered triples
$\{(u_i, v_i, w_i)|i\in[d, d+m-1]\}$ with
$\bigcup_{i=1}^{m}\{u_i, v_i, w_i\}=[d, d+3m]\setminus\{k+d+m-1\}$ and $u_i+v_i=w_i$.
We also have a collection of ordered triples
$\{(u_i, v_i, -w_i), (-u_i, -v_i, w_i)|i\in[d, d+m-1]\}$ with
$\bigcup_{i=d}^{d+m-1}(\{u_i, v_i, -w_i\})\bigcup\{-u_i, -v_i, w_i\})=[-d-3m,-d]\cup[d, d+3m]\setminus\{k+d+m-1, -k-d-m+1\}$ and $u_i+v_i=w_i$.

Further, let
\begin{center}
$x_0=k+d+m-1, y_0=-(k+d+m-1), z_0=0$,   \ \  \ \ \ \ \ \ \ \ \ \  \\
$x_i=u_{i+d-1}, y_i=v_{i+d-1}, z_i=-w_{i+d-1},i\in[1, m]$,  \ \  \ \ \ \   \ \ \ \ \ \ \ \ \ \\
$x_{-i}=-u_{i+d-1}, y_{-i}=-v_{i+d-1}, z_{-i}=w_{i+d-1},i\in[1, m].$  \ \ \ \ \ \ \ \ \\
\end{center}
and
\begin{center}
 \ \ \ \ \ \ \ \ \  \ \ \ \ \ \ \ \ \   \ \ \ \  \  \ \ \  \ \  \ \ \ \ \ \ \ \ \  \ \ \ \ \ \ \ \ \   $T_i=(x_i,y_i,z_i), i\in[-m,m]$.  \ \ \ \ \ \ \ \ \  \ \ \ \ \ \ \ \ \  \ \ \ \ \ \ \ \ \  \ \ \ \  \ \ \ \  \ \ \ \ \ \ \ \     (\romannumeral1)
\end{center}
We get a partition
$[d,d+3m]\cup [-d-3m,-d] \cup \{0\}= \bigcup_{i\in [-m,m]}\{x_i,y_i,z_i\}.$
Clearly, each triple has the property that the sum of three members is zero. So we have

\begin{lem} \label{triple}
If there is an  $\mathcal{L}^k_{d,m}$ then there is a partition  of
the set $[d,d+3m]\cup [-d-3m,-d] \cup \{0\}$ into triples each having the property that the  sum of the members is  zero.
\end{lem}

A QMR$(m,n)$ is denoted by QMR$^*(m,n)$ if the row sum  is zero and the column sum is also zero.
A QMR$(m,n,r)$ is denoted by QMR$^*(m,n,r)$ if it consists of  $r$ QMR$^*(m,n)$s  and if they are supposed then the sum of $r$ numbers in each entry is zero.

  Suppose that there is an  $\mathcal{L}^k_{4,m}$. Let $T_i (i\in[-m,m])$ be the triples  given by  (\romannumeral1).
For any $T_i(i\in[-m,m])$  we define a  $3\times3\times3$ array $M_i=(M_{i,1},M_{i,2},M_{i,3})$, where $M_{i,1}, M_{i,2}, M_{i,3}$ are given below.
\vskip5pt
 \begin{center}
 {\renewcommand\arraystretch{1}
\setlength{\arraycolsep}{2.0pt}
\small
\ \ \ \ \ \ \ \ $\left(
      \begin{array}{ccc}
 9x_i+1&9z_i-4&9y_i+3\\
 9z_i-3&9y_i+4&9x_i-1\\
 9y_i+2&9x_i&9z_i-2\\
\end{array}  \right),
       $
        $\left(
      \begin{array}{ccc}
 9y_i-3&9x_i+4&9z_i-1\\
 9x_i+2&9z_i&9y_i-2\\
 9z_i+1&9y_i-4&9x_i+3\\
\end{array}\right),
$                $\left(
      \begin{array}{ccc}
 9z_i+2&9y_i&9x_i-2\\
 9y_i+1&9x_i-4&9z_i+3\\
 9x_i-3&9z_i+4&9y_i-1\\
\end{array} \right).$}\ \ \ \ \ \ \ \ (\romannumeral2)
         \end{center}
         \vskip5pt
 It is readily verified that each  $M_i$ is a  QMR$^*(3,3,3)$,  $i\in [-m,m]$, and
the entries of all $M_i(i\in [-m,m])$ run over the set $S=[-4,4]\cup [32,40+27m] \cup [-40-27m,-32]$. Clearly, $|S|=27(2m+1)
$.  So we have the following.

\begin{lem} \label{triple-2}
If there is an  $\mathcal{L}^k_{4,m}$ then there  are $2m+1$ QMR$^*(3,3,3)$s with the entries run over the set
 $S$.
\end{lem}

We should point out that each   $-M_i(i\in [1,m])$ is also a QMR$^*(3,3,3)$ and the element set of $-M_i$ is exactly
the  element set of $M_{-i}$.  Thus $M_0, \pm M_1, \cdots, \pm M_m$ are also the  QMR$^*(3,3,3)$s with the property mentioned in Lemma \ref{triple-2}.
Further, the sum of the elements in the main diagonal of  each $\pm M_{i,2} (i\in [0,m])$ is zero.

\vskip10pt

\noindent {\bf Example 1}  The sequence $L=(l_1, l_2, \cdots, l_{11})$  listed below is an $\mathcal{L}^6_{4,5}$.
 \begin{center}
$
      \begin{array}{ccccccccc ccccccccc ccccccccc   }
l_1&l_2&l_3&l_4&l_5&l_6&l_7&l_8&l_9&l_{10}&l_{11}\\
8&6&4&7&5&&4&6&8&5&7
   \end{array}
 $
\end{center}
By Lemma \ref{triple} there is a partition of the set $[4,4+3m]\cup [-4-3m,-4] \cup \{0\}$ into  triples as follows.
\begin{center}
\ \ $T_0=(14,-14,0)$,\ \ \ \ \ \ \ \  \ \ \ \   \ \  \ \ \ \ \  \ \ \ \ \ \ \  \ \ \ \ \ \ \ \\
$T_1=(4,11,-15)$,\ \  $T_{-1}=(-4,-11,15)$,\\
$T_2=(5,13,-18)$,\ \  $T_{-2}=(-5,-13,18)$,\\
$T_3=(6,10,-16)$,\ \  $T_{-3}=(-6,-10,16)$,\\
$T_4=(7,12,-19)$,\ \  $T_{-4}=(-7,-12,19)$,\\
$T_5=(8,9,-17)$, \ \ \ $T_{-5}=(-8,-9,17)$.\  \ \
\end{center}
By Lemma \ref{triple-2} we have a QMR$^*(3,3,3)$, $M_0=(M_{01}, M_{02}, M_{03})$,  where
 \begin{center}
 {\renewcommand\arraystretch{1}
\setlength{\arraycolsep}{1.6pt}
\footnotesize
$M_{0,1}=\left(
      \begin{array}{ccc}
127&-4&-123\\
-3&-122&125\\
-124&126&-2\\
   \end{array}\right),
       $ \ \
        $M_{0,2}=\left(
      \begin{array}{ccc}
-129&130&-1\\
128&0&-128\\
1&-130&129\\
   \end{array}\right),
$\ \
                $M_{0,3}=\left(
      \begin{array}{ccc}
2&-126&124\\
-125&122&3\\
123&4&-127\\
   \end{array}\right).
 $}
         \end{center}
 One can easily list the other 10 QMR$^*(3,3,3)$s $\pm M_i, i\in[1,5]$.   The entries of these 11 QMR$^*(3,3,3)$s run over the set $[-4,4]\cup [32,175]\cup [-175,-32]$.

\vskip5pt

\section{Proof of  Theorem 1.1}

Let $n\equiv\pm3\pmod{18}, n\geq21$. Denote $n=18u\pm3$ and $\lambda=(n^2-1)/2$.
We shall construct an ASMS$(n)$
$A=(a_{ij}), i,j\in [1,n]$ over the set $[-\lambda,\lambda]$.

  Suppose that there is an  $\mathcal{L}^k_{4,m}$.  Let $T_i=(x_i,y_i,z_i) (i\in[-m,m])$ be the ordered triples defined as (\romannumeral1) and
 let $M_0, M_i,-M_i, i\in [1,m]$   be the   $2m+1$ QMR$^*(3,3,3)$s defined as (\romannumeral2). The entries of these  QMR$^*(3,3,3)$s
    run over  the set $S=[-4,4]\cup [32,40+27m] \cup [-40-27m,-32]$ by Lemma \ref{triple-2}. Note that $|S|=27(2m+1)$.

Let $m=6u^2\pm2u-3$. Then $27(2m+1)=n^2-144$, i.e., $m=(n^2-171)/54$. So, $S\subset[-\lambda,\lambda]$,  $|[-\lambda,\lambda]|-|S|=144$.

Write $w=\frac{n}{3}=6u\pm1$. Partition  $A$ into submatrices of order 3,
$A=(A_{st}), \ s,t\in [1,w].$
We use binary Cartesian product to denote the index set of $A_{st}$, i.e.,   $(s,t)\in [1,w]\times[1,w]$. Divide $ [1,w]\times[1,w]$ into four parts as follows.

$V_1=[1,4]\times[1,4]\cup\{(w,w)\}\setminus\{(4,4)\},$
 \ \ \     $V_2=[1,2]\times[5,w],$

$V_3=[5,w]\times[1,2],$\ \  \ \ \ \ \ \ \ \ \ \ \ \ \  \ \ \ \ \ \   \ \ \ \ \ \ \ \  \ \
$V_4=[3,w]\times[3,w]\setminus \{(3,3),(3,4),(4,3),(w,w)\}.$\\

By calculation we have $[-\lambda,\lambda]\setminus S=[5,31]\cup[-31,-5]\cup [\lambda-44,\lambda]\cup[-\lambda,44-\lambda]$ and $|[-\lambda,\lambda]\setminus S|=144$. We shall fill these 144 numbers in $V_1$  and the rest numbers of $[-\lambda,\lambda]$  in  $V_2, V_3$ and $V_4$ in a proper way.
The construction follows the following four steps.

\vskip5pt
{\bf Step 1} The following table gives the blocks $A_{s\times t}$ with indices  $(s,t)\in V_1\setminus\{(w,w)\}$.
 \begin{center}
 {\renewcommand\arraystretch{1.0}
\setlength{\arraycolsep}{2.2pt}
\footnotesize
 $      \begin{array}{|c|c|c|c|c|c|c|c|c|c|c|c|  }
 \hline
17&\lambda-17&-25&29&24-\lambda&\lambda-20&12-\lambda&\lambda-26&6-\lambda&\lambda-22&42-\lambda&-20\ \ \\ \hline
17-\lambda&-17&25&-29&\lambda-24&20-\lambda&\lambda-12&26-\lambda&\lambda-6&22-\lambda&\lambda-42&20\\ \hline
-22&22&14&\lambda-9&13-\lambda&\lambda-29&1-\lambda&37-\lambda&\lambda-27&11-\lambda&\lambda-21&10\\ \hline
30&-30&9-\lambda&-14&\lambda-13&29-\lambda&\lambda-1&\lambda-37&27-\lambda&\lambda-11&21-\lambda&-10\\ \hline
-23&23&\lambda-5&5-\lambda&-11&8-\lambda&\lambda-16&\lambda-15&34-\lambda&\lambda-4&10-\lambda&-6\\ \hline
26&-26&\lambda-38&38-\lambda&\lambda-8&11&16-\lambda&15-\lambda&\lambda-34&4-\lambda&\lambda-10&6\\ \hline
\lambda&0-\lambda&36-\lambda&\lambda-36&3-\lambda&\lambda-3&-5&\lambda-39&44-\lambda&-18&-13&31\\ \hline
\lambda-32&32-\lambda&\lambda-41&41-\lambda&30-\lambda&\lambda-30&7-\lambda&21&\lambda-28&-9&28&-19\\ \hline
35-\lambda&\lambda-35&25-\lambda&\lambda-25&\lambda-14&14-\lambda&\lambda-2&18-\lambda&-16&27&-15&-12\\ \hline
\lambda-43&43-\lambda&\lambda-31&31-\lambda&40-\lambda&\lambda-40&18&13&-31&&&\\ \hline
19-\lambda&\lambda-19&23-\lambda&\lambda-23&\lambda-33&33-\lambda&9&-28&19&&&\\ \hline
-24&24&8&-8&-7&7&-27&15&12&&&\\\hline
   \end{array}$}
         \end{center}
         \vskip5pt
Taking $A_{ww}=-A_{33}$ together with the above table  gives all the blocks $A_{s\times t}$ with indices  $(s,t)\in V_1$.  Since $n\geq21$ the element set of the blocks $A_{s\times t}$ with indices  $(s,t)\in V_1$  is exactly $[-\lambda,\lambda]\setminus S$.

 \vskip5pt
{\bf Step 2}   Performing   row permutation
$\left(
      \begin{smallmatrix}
1&2&3&4&5&6 \\
2&4&6&1&3&5
   \end{smallmatrix}
      \right)$
 and column  permutation
$\left(
      \begin{smallmatrix}
1&2&3 \\
2&3&1
   \end{smallmatrix}
      \right)$
on the matrix
$\left(
      \begin{smallmatrix}
M_{i,2} \\
-M_{i,2}
   \end{smallmatrix}
      \right)$
gives a $6\times3$ matrix as follows.
 \begin{center}
 {\renewcommand\arraystretch{0.8}
\setlength{\arraycolsep}{2.8pt}
\small
        $P_i=\left(
      \begin{array}{ccc}
      1-9z_i&3-9y_i&-9x_i-4\\
9z_i-1&9y_i-3&9x_i+4\\
2-9y_i& -9x_i-2&-9z_i\\
9y_i-2& 9x_i+2&9z_i\\
-9x_i-3& -9z_i-1&4-9y_i\\
9x_i+3& 9z_i+1&9y_i-4\\
\end{array}\right), i\in[1,w-4].
$}
         \end{center}
    Taking
     \begin{center}
      $\left(
      \begin{smallmatrix}
A_{1,i+4}  \\
A_{2,i+4}
   \end{smallmatrix}
      \right)=P_i, i\in[1,w-4]$
     \end{center}
gives all the blocks $A_{s,t}$ with $(s,t)\in V_2$.

\vskip5pt
{\bf Step 3} Performing   column  permutation
 $\left(
      \begin{smallmatrix}
1&2&3&4&5&6 \\
1&3&5&2&4&6
   \end{smallmatrix}
      \right)$  on the matrix  $(M_{i,3}^T, -M_{i,3}^T)$ gives a $3\times6$ matrix as follows.
\begin{center}
 {\renewcommand\arraystretch{0.8}
\setlength{\arraycolsep}{2.8pt}
\small
        $Q_i=\left(
      \begin{array}{cccccc}
 9z_i+2& -9z_i-2&9y_i+1&-9y_i-1& 9x_i-3&3-9x_i\\
 9y_i&-9y_i&9x_i-4&4-9x_i&9z_i+4&-9z_i-4\\
 9x_i-2& 2-9x_i&9z_i+3&-9z_i-3&9y_i-1&1-9y_i\\
\end{array}\right), i\in[1,w-4].
$}
         \end{center}
Taking
\begin{center}
$(A_{i+4, 1},  A_{i+4, 2})=Q_i$, $i\in[1,w-4]$
\end{center}
gives all the blocks $A_{s,t}$ with $(s,t)\in V_3$.

   {\bf Step 4}    Let
          $V_{4,1}=\{(w-2,w),(w-1,w-1),(w,w-2)\}$, and
    \begin{center}
     $V_{4,2}=\{(h-1,h+1),(h,h),(h+1,h-1), h=4,5,\cdots,w-2\}$,
    \end{center}
   and   $V_{4,3}=V_4\setminus (V_{4,1}\cup V_{4,2})$.
Taking
         \begin{center}
     \ \ \ \ \ \ \ \  \ \ \ \ \ \ \ \    \ \ \ \ \ \ \ \  \ \ \ \ \    $A_{w-2,w}=M_{0,1}$,  $A_{w-1,w-1}=M_{0,2}$, $A_{w-2,w}=M_{0,3}$.    \ \ \ \ \ \ \ \ \ \ \ \   \ \ \ \ \ \ \ \  \ \ \ \ \ \ \ \   (\romannumeral3)
         \end{center}
          gives  the blocks $A_{s,t}$ with $(s,t)\in V_{4,1}$.
Taking
\begin{flushleft}\ \ \ \ \ \ \ \ \ \ \ \
$A_{h-1,h+1}=M_{m-(h-4),1}$, $A_{h,h}=M_{m-(h-4),2}$, $A_{h+1,h-1}=M_{m-(h-4),3}$, \ \ \ \ \ \ \ \ \ \ \ \  \ \ \ \ \ \ \ \ (\romannumeral4)\\\ \ \ \ \ \ \ \ \ \ \ \
$A_{h,h+2}=-M_{m-(h-4),1}$, $A_{h+1,h+1}=-M_{m-(h-4),2}$, $A_{h+2,h}=-M_{m-(h-4),3}$. \ \ \ \ \ \ \ \ \ \ \ \ (\romannumeral5)
\end{flushleft}
gives the blocks $A_{s,t}$ with $(s,t)\in V_{4,2}$.
Note that $w-5$  is  even since $n\equiv\pm3\pmod{18}$.

The remaining $M_{i,j}$s are put in $V_{4,3}$ so that they  satisfy the following property.
\begin{center}
 \ \ \ \ \ \ \ \  \ \ \ \ \ \ \ \  \ \ \ \ \ \ \ \  \ \ \ \ \ \ \ \  \ \ \ \ \ \ \ \ if $A_{s,t}=M_{i,j}$ then $A_{t,s}=-M_{i,j}$. \ \ \ \ \ \ \ \ \ \ \ \  \ \ \ \ \ \ \ \  \ \ \ \ \ \ \ \  \ \ \ \ \ \ \ \  \ \ \ \ \ \ \ \ (\romannumeral6)
\end{center}

The following table is used to show the positions of $P_i, Q_i$ and the blocks in $V_{4,1}$ and $V_{4,2}$.

\begin{center}
 {\renewcommand\arraystretch{1.2}
 \setlength{\arraycolsep}{1.2pt}
 \small
        $
      \begin{array}{|c|c|c|c|c|c|c|c|c|c|c|c|}
      \hline
\ \ \ \ \ \ \ \ &\ \ \ \ \ \ \ \ &&& \multirow{2}{1cm}{$P_1$}&\multirow{2}{1cm}{$P_2$}&\multirow{2}{1cm}{$\cdots$}&\multirow{2}{1cm}{$P_{w-6}$}&\multirow{2}{1cm}{$P_{w-5}$}&\multirow{2}{1cm}{$P_{w-4}$}\\\cline{1-4}
&&&&&&&&&\\\hline\cline{3-10}
&\multicolumn{1}{c|}{}\vline&&&M_{m,1}&&\cdots&&&\\  \hline
&\multicolumn{1}{c|}{}\vline&&M_{m,2}&&-M_{m,1}&\cdots&&&\\  \hline
\multicolumn{2}{|c|}{Q_1}\vline&M_{m,3}&&-M_{m,2}&&\cdots&&&\\  \hline
\multicolumn{2}{|c|}{Q_2}\vline&&-M_{m,3}&&&\cdots&&&\\  \hline
\multicolumn{2}{|c|}{ \vdots } \vline&\vdots&\vdots&\vdots&\vdots&\vdots&\vdots&\vdots&\vdots\\  \hline
\multicolumn{2}{|c|}{Q_{w-6}}\vline&&&&&\cdots&&&M_{0,1}\\  \hline
\multicolumn{2}{|c|}{Q_{w-5}}\vline&&&&&\cdots&&M_{0,2}&\\  \hline
\multicolumn{2}{|c|}{Q_{w-4}}\vline&&&&&\cdots&M_{0,3}&&\\  \hline
\end{array}
$
 }
         \end{center}

 \begin{lem} $P_i$, $Q_i$$(i\in[1,w-4])$ have the following properties.

\emph{(\uppercase\expandafter{\romannumeral1})}  $P_i$ is a QMR$^*(6,3)$, the  last four rows of $P_i$ form a QMR$^*(4,3)$, and the   last two rows of $P_i$ form a QMR$^*(2,3)$.

\emph{(\uppercase\expandafter{\romannumeral2})}  $Q_i$ is a QMR$^*(3,6)$s, the last four columns  of $Q_i$ form  a QMR$^*(3,4)$, and  the last two columns  of $Q_i$ form  a QMR$^*(3,2)$.
\end{lem}
 \begin{proof}
The conclusions follow from a direct calculation.
  \end{proof}

A matrix $A$   constructed by  taking the above four steps has the following properties.

\begin{lem} $A$ has the following useful properties.

\emph{(\uppercase\expandafter{\romannumeral3})}  Each block $A_{s,t}$ with $(s,t)\in [3,w]\times[3,w]$ is   a QMR$^*(3,3)$.

\emph{(\uppercase\expandafter{\romannumeral4})}  The sum of the main diagonal of $A_{s,s}$ is zero for each $s\in[3,w]$.

\emph{(\uppercase\expandafter{\romannumeral5})} The equality
$\sum_{i\in[7,h-7]} a_{i,h-i}=0$ holds
for each  $h\in[20,n+7]$.
\end{lem}
 \begin{proof}
(\uppercase\expandafter{\romannumeral3}) $A_{3,3},A_{3,4},A_{4,3}$ and $A_{w,w}$  are  QMR$^*(3,3)$s by using direct calculation.
For $(s,t)\in V_4$ the blocks $A_{s,t}$  are all QMR$^*(3,3)$s by
(\romannumeral3)-(\romannumeral6).

(\uppercase\expandafter{\romannumeral4}) Clearly, the  sum of the main diagonal of $A_{s,s}$ is zero when $s=3$ and $s=w$.
By (\romannumeral2)  the  sum of the main diagonal of each QMR$^*(3,3)$ $\pm M_{i,2} (i\in [0,m])$ is zero. By (\romannumeral3) and (\romannumeral4)
 the sum of the elements in the main diagonal of $A_{s,s}$ is zero when $s\in[4,w-1]$ .

(\uppercase\expandafter{\romannumeral5})  Let $h\in[20,n+7]$. For   $i\in[7,h-7]$,
 let $s_0=\lceil\frac{i}{3}\rceil, t_0=\lceil\frac{h-i}{3}\rceil$, where $\lceil a\rceil$ is the smallest integer $x$ such that $x\geq a$. Then
 $(s_0,t_0)\in V_4\cup\{(3,4), (4,3)\}$ and $a_{i,h-i}$ is an element of  $A_{s_0,t_0}$.

  If $(s_0,t_0)\in V_{4,1}\cup V_{4,2}$ then
 $s_0+t_0$ is even, denoted by $2d_0$. So $a_{i,h-i}$ belongs to one of $A_{d_0-1,d_0+1}$, $A_{d_0,d_0}$ and $A_{d_0+1,d_0-1}$.
  Since $M_{0}, M_{i}, -M_{i}(i\in [1,m])$ are QMR$^*(3,3,3)$s
 the sum of the $(i\pmod 3,(h-i)\pmod3)$-entries of the  blocks $A_{d_0-1,d_0+1}$, $A_{d_0,d_0}$, $A_{d_0+1,d_0-1}$ is zero by  (\romannumeral3), (\romannumeral4), (\romannumeral5).

 If $(s_0,t_0)\not\in V_{4,1}\cup V_{4,2}$ then
the sum of the $(i\pmod 3,(h-i)\pmod3)$-entries of the blocks  $A_{s_0,t_0}$ and  $A_{t_0,s_0}$ is zero  by (\romannumeral6).
The proof is completed.
   \end{proof}

\begin{thm} \label{mainthm} Let $n\equiv\pm3\pmod{18}$ and $m=(n^2-171)/54$. If there is an  $\mathcal{L}^k_{4,m}$ then there is an ASMS$(n)$.
\end{thm}
 \begin{proof} Let $n$ and $m$ be as assumption.  Suppose that there is an  $\mathcal{L}^k_{4,m}$.  Let $T_i=(x_i,y_i,z_i) (i\in[-m,m])$ be the ordered triple defined as (\romannumeral1) and
 let $M_0, M_i,-M_i, i\in [1,m]$   be the   $2m+1$ QMR$^*(3,3,3)$s defined as (\romannumeral2).  The element  set of the matrix  $A$ under the above construction is exactly $[-\lambda,\lambda]$. We now prove that   $A$   is an ASMS$(n)$.

The blocks $A_{s,t}$ with $(s,t)\in [1,4]\times[1,4]$ form a   QMR$^*(12,12)$ by using direct calculation and (\uppercase\expandafter{\romannumeral3}).
The blocks $A_{s,t}$ with $(s,t)\in [1,2]\times[5,w]$
form a   QMR$^*(6,n-12)$ by (\uppercase\expandafter{\romannumeral1}), and the blocks $A_{s,t}$ with $(s,t)\in[5,w]\times[1,2]$
form a   QMR$^*(6,n-12)$ by (\uppercase\expandafter{\romannumeral2}). Combining with (\uppercase\expandafter{\romannumeral3}) we know that $A$ is a QMR$^*(n,n)$.
The sum of the first six elements $17, -17, 14, -14, -11, 11$ in the main diagonal of $A$  from the upper left to the lower right  is zero. Combining with (\uppercase\expandafter{\romannumeral4}) the main diagonal of $A$ has zero sum.
 The first six elements of back diagonal of $A$ from upper right to lower left
 are
 $-9x_{w-4}-4,\ 9y_{w-4}-3,\ 2-9y_{w-4}, \ 9z_{w-5},\ -9z_{w-5}-1,\ 9x_{w-5}+3$, and the last six elements  are
 $9x_{w-4}-2,\ -9y_{w-4},\ 9y_{w-4}+1,\ -9z_{w-5}-3,\ 9z_{w-5}+4,\ 3-9x_{w-5}.$
 The sum of the above 12 numbers is 0. Combining with (\uppercase\expandafter{\romannumeral5}) the back diagonal of $A$ has zero sum.
 Thus $A$ is an MS$(n)$.

Now we prove that $A$ contains a GMS$(k)$ for  $3\leq k\leq n-2$.
Since $M_{0,2}$ is a GMS$(3)$, $A_{w-1,w-1}$ is a GMS$(3)$ by (\romannumeral3).
 The matrix $(a_{ij})$ with $(i,j)\in [1,4]\times [8,11]$ is a GMS(4) by direct calculation.
 Let $C=(a_{ij}), i,j\in [1,n-3]$. Then $C$ is a GMS$(n-3)$.
 In fact, the matrix $(a_{ij})(i,j\in[1,12])$ is a  QMR$^*(12,12)$ with its main diagonal having zero sum.
 By (\uppercase\expandafter{\romannumeral1})-(\uppercase\expandafter{\romannumeral4}) $C$ is a QMR$^*(n-3,n-3)$ with its main diagonal having zero sum.
 For the $n-3$   elements in the back diagonal of $C$, $Q_{w-5}$ has three elements, $9x_{w-5}-2,-9y_{w-5},9y_{w-5}+1$.
 $Q_{w-6}$ has three elements, $-9z_{w-6}-3,9z_{w-6}+4, 3-9x_{w-6}$. $P_{w-5}$ has three elements $-9x_{w-5}-4, 9y_{w-5}-3,2-9y_{w-5}$.
 $P_{w-6}$ has three elements $9z_{w-6},  -9z_{w-6}-1,9x_{w-6}+3$. The sum of the above 12 elements is zero. Combining with
 (\uppercase\expandafter{\romannumeral5}) we know that the sum of the elements of the back diagonal of $C$ is zero. So, $C$ is a  GMS$(n-3)$.

Let $h\in [5,n-2]\setminus\{n-3\}$. We have \\
  (1) if $h\equiv0\pmod3$ then the matrix $(a_{ij})_{h\times h}$ with  $(i,j)\in [7,6+h]\times[7,6+h]$  is a GMS$(h)$;\\
  (2) if $h\equiv1\pmod3$ then the matrix $(a_{ij})_{h\times h}$ with $(i,j)\in [3,2+h]\times[3,2+h]$ is a GMS$(h)$;\\
  (3) if $h\equiv2\pmod3$ then the matrix $(a_{ij})_{h\times h}$ with $(i,j)\in [5,4+h]\times[5,4+h]$ is a GMS$(h)$.

In fact, for the case (1),   $h\in\{6, 9, \cdots, n-6\}$. By (\uppercase\expandafter{\romannumeral3})-(\uppercase\expandafter{\romannumeral5})    $(a_{i,j})$ with  $(i,j)\in [7,6+h]\times [7,6+h]$  is a GMS$(h)$.

 For Case (2),  $h\in\{7, 10, \cdots, n-2\}$. Let  $B(h)=(a_{ij})_{h\times h}$ where $(i,j)\in [3,2+h]\times [3,2+h]$.
  It is readily verified that  $B(h)$ is a GMS$(h)$ for $h=7,10$ by using direct calculation together with properties (\uppercase\expandafter{\romannumeral3}) and (\uppercase\expandafter{\romannumeral4}).
  Let $h\geq13$.
  By direct  calculation we have $\sum_{j=3}^{12}a_{i,j}=0,i\in[3,6]$ and $\sum_{i=3}^{12}a_{i,j}=0,j\in[3,6]$.
Combining with (\uppercase\expandafter{\romannumeral1})-(\uppercase\expandafter{\romannumeral3}) we know that $B(h)$ is a QMR$^*(h,h)$.
The first four elements of the main diagonal of $B(h)$ from upper left have zero sum.
By (\uppercase\expandafter{\romannumeral4}) we know that the main diagonal of $B(h)$ has zero sum.
Now we check the back diagonal of  $B(h)$.
Note that $h\geq13$. The first four elements and the last four elements of the back diagonal of $B(h)$ from the upper right to the lower left are
 $-9z_{f},9x_{f}+2,-9x_{f}-3,d_1,d_2,9x_{f}-3,4-9x_{f},9z_{f}+3$, where $f=\frac{h-10}{3}$.
 If $h\geq16$ then $d_1=1-9y_{f-1}$ and $d_2=9y_{f-1}-4$.  The sum of these eight numbers is zero.
 If $h=13$ then $d_1=6$ and $d_2=7$. The sum of these eight numbers is $16$. But the remaining five elements are $-13,-9,-16,13,9$, which have the sum $-16$.
 Combining with (\uppercase\expandafter{\romannumeral5}) we know that the back diagonal of $B(h)$ has zero sum for $h\in[13,n-2]$ and $h\equiv1\pmod3$. Thus $B(h)$ is a GMS$(h)$.

Now we consider  Case (3). Clearly, $h\in\{5, 8, \cdots, n-4\}$.  Let   $E(h)=(a_{ij})_{h\times h}$, where $(i,j)\in [5,4+h]\times [5,4+h]$.
 It is readily verified that  $E(h)$ is a GMS$(h)$ for $h=5,8$ by using direct calculation together with properties (\uppercase\expandafter{\romannumeral3}) and (\uppercase\expandafter{\romannumeral4}).
  Let $h\geq11$. By direct calculation we have $\sum_{j=5}^{12}a_{i,j}=0,i\in[5,6]$ and $\sum_{i=5}^{12}a_{i,j}=0,j\in[5,6]$.
  Combining with (\uppercase\expandafter{\romannumeral1})-(\uppercase\expandafter{\romannumeral4}) we know that $E(h)$ is a  QMR$^*(h,h)$, and the main diagonal of $E(h)$ has zero sum.
 The first two elements and the last two elements of the back diagonal of $E(h)$ from the  upper right to lower left are
 $4-9y_{f},9z_{f}+1,-9z_{f}-4,9y_{f}-1$, where $f=\frac{h-8}{3}$. The sum of the four numbers is zero. Combining with (\uppercase\expandafter{\romannumeral5}) we know that the back diagonal of $E(h)$ has zero sum. Thus $E(h)$ is a GMS$(h)$.
Therefore $A$ is an ASMS$(n)$. The proof is completed.
\end{proof}

\noindent {\bf Example 2}   Following  the 4 steps   above  we get an   ASMS$(21)$ below.
 \begin{center}
 {\renewcommand\arraystretch{0.9}
\setlength{\arraycolsep}{0.4pt}
\footnotesize $      \begin{array}{|cccc ccccc ccc|ccc|ccc|ccc|ccc| }
      \hline
17&203&-25&29&-196&200&-208&\underline{194}&\underline{-214}&\underline{198}&\underline{-178}&-20&136&-96&-40&163&-114&-49&145&-87&-58\\
-203&-17&25&-29&196&-200&208&\underline{-194}&\underline{214}&\underline{-198}&\underline{178}&20&-136&96&40&-163&114&49&-145&87&58\\
-22&22&14&211&-207&191&-219&\underline{-183}&\underline{193}&\underline{-209}&\underline{199}&10&-97&-38&135&-115&-47&162&-88&-56&144\\
30&-30&-211&-14&207&-191&219&\underline{183}&\underline{-193}&\underline{209}&\underline{-199}&-10&97&38&-135&115&47&-162&88&56&-144\\
-23&23&215&-215&-11&-212&204&205&-186&216&-210&-6&-39&134&-95&-48&161&-113&-57&143&-86\\
26&-26&182&-182&212&11&-204&-205&186&-216&210&6&39&-134&95&48&-161&113&57&-143&86\\\cline{7-21}
220&-220&-184&184&-217&\multicolumn{1}{c|}{217}&-5&181&\multicolumn{1}{c|}{-176}&-18&-13&31&73&-157&84&37&-139&102&55&-148&93\\
188&-188&179&-179&-190&\multicolumn{1}{c|}{190}&-213&21&\multicolumn{1}{c|}{192}&-9&28&-19&-156&85&71&-138&103&35&-147&94&53\\
-185&185&-195&195&206&\multicolumn{1}{c|}{-206}&218&-202&\multicolumn{1}{c|}{-16}&27&-15&-12&83&72&-155&101&36&-137&92&54&-146\\\cline{7-21}
177&-177&189&-189&-180&\multicolumn{1}{c|}{180}&18&13&\multicolumn{1}{c|}{-31}&78&76&-154&46&-166&120&-73&157&-84&64&-175&111\\
-201&201&-197&197&187&\multicolumn{1}{c|}{-187}&9&-28&\multicolumn{1}{c|}{19}&74&-153&79&-165&121&44&156&-85&-71&-174&112&62\\
-24&24&8&-8&-7&\multicolumn{1}{c|}{7}&-27&15&\multicolumn{1}{c|}{12}&-152&77&75&119&45&-164&-83&-72&155&110&63&-173\\\hline
-133&133&100&-100&33&\multicolumn{1}{c|}{-33}&-151&81&\multicolumn{1}{c|}{70}&-46&166&-120&-78&-76&154&105&67&-172&127&-4&-123\\
99&-99&32&-32&-131&\multicolumn{1}{c|}{131}&82&68&\multicolumn{1}{c|}{-150}&165&-121&-44&-74&153&-79&65&-171&106&-3&-122&125\\
34&-34&-132&132&98&\multicolumn{1}{c|}{-98}&69&-149&\multicolumn{1}{c|}{80}&-119&-45&164&152&-77&-75&-170&104&66&-124&126&-2\\\hline
-160&160&118&-118&42&\multicolumn{1}{c|}{-42}&-37&139&\multicolumn{1}{c|}{-102}&151&-81&-70&-105&-67&172&\underline{-129}&\underline{130}&\underline{-1}&-169&108&61\\
117&-117&41&-41&-158&\multicolumn{1}{c|}{158}&138&-103&\multicolumn{1}{c|}{-35}&-82&-68&150&-65&171&-106&\underline{128}&\underline{0}&\underline{-128}&109&59&-168\\
43&-43&-159&159&116&\multicolumn{1}{c|}{-116}&-101&-36&\multicolumn{1}{c|}{137}&-69&149&-80&170&-104&-66&\underline{1}&\underline{-130}&\underline{129}&60&-167&107\\\hline
-142&142&91&-91&51&\multicolumn{1}{c|}{-51}&-55&148&\multicolumn{1}{c|}{-93}&-64&175&-111&2&-126&124&169&-108&-61&5&-181&176\\
90&-90&50&-50&-140&\multicolumn{1}{c|}{140}&147&-94&\multicolumn{1}{c|}{-53}&174&-112&-62&-125&122&3&-109&-59&168&213&-21&-192\\
52&-52&-141&141&89&\multicolumn{1}{c|}{-89}&-92&-54&\multicolumn{1}{c|}{146}&-110&-63&173&123&4&-127&-60&167&-107&-218&202&16\\\hline
   \end{array}.$}
         \end{center}
The underlined  are a GMS$(3)$ and a GMS$(4)$.  One can easily find the GMS$(k)$ for $5\leq k\leq19$ by Theorem \ref{mainthm}.

\vskip 5pt

 Theorem 7.2 in \cite{Linek2004} provided the following.
\begin{lem} \emph{(\cite{Linek2004})} \label{seq}
  There exists an $\mathcal{L}^k_{4,m}$ whenever $(m,k)\equiv(0,1),(1,0),(2,0),(3,1)\pmod{(4,2)}$, on condition that $m\geq5$ and $\frac{m}{2}(7-m)+1\leq k \leq\frac{m}{2}(m-3)+1$.
 \end{lem}

\vskip 5pt
\noindent {\bf Theorem \ref{1.1}}
 There exists an  ASMS$(n)$ for $n\equiv\pm3\pmod{18}$.
\vskip 5pt
  \begin{proof}
The necessary condition for an  ASMS$(n)$ to exist is $n\geq5$ by definition. Let $n\equiv3,15\pmod{18}$. An ASMS$(15)$ was given by Abe (\cite{Abe}). For $n\geq21$, let $m=(n^2-171)/54$, then $m\geq5$. Denote  $n=3(6u\pm1)$,  we have $m=6u^2\pm2u-3=4u^2+2u(u\pm1)-3\equiv1\pmod4$.
 By Lemma \ref{seq} there exists an $\mathcal{L}^k_{4,m}$ for some $k\equiv0\pmod2$. So, there exists  an  ASMS$(n)$  by Theorem \ref{mainthm}.
The proof is completed.
\end{proof}

\noindent\textbf{Concluding remarks}

The existence of an ASMS$(n)$ for $n\equiv\pm3\pmod{18}$ is completely solved. To give a complete solution of  Abe's conjecture  2.12, the following cases should be considered: (1) $n\equiv9\pmod{18}$; (2) $n\equiv\pm\pmod{6}$.

 \vskip 10pt
 \noindent\textbf{Acknowledgements} The authors would like to thank Professor Zhu Lie of Suzhou University for  his encouragement and many helpful suggestions
 especially on the application of extend Langford sequences. Thank Professor Chen Kejun and Professor Pan Fengchu for careful reading and some suggestions.  

\end{document}